\documentclass[12pt,letterpaper,twoside]{amsart}
\usepackage{amsmath,amssymb}
\usepackage{amsthm}
\usepackage[abbrev]{amsrefs}
\usepackage{latexsym}

\numberwithin{equation}{section}
\theoremstyle{definition} 
\newtheorem{thm}{Theorem}[section]
\newtheorem{prop}[thm]{Proposition}
\newtheorem{lem}[thm]{Lemma}
\newtheorem{cor}[thm]{Corollary}
\newtheorem{dfn}[thm]{Definition}

\newtheorem{rem}[thm]{Remark}

\newcommand{\set}[1]{\{\,{#1}\,\}}
\newcommand{\Image}{\mathop{\rm Im}}
\newcommand{\R}{\mathbb R}

\newcommand{\Sn}{{S^n(1)}}
\newcommand{\Srn}{{S^n(\sqrt n)}}

\newcommand{\RPn}{{\mathbb RP^n}}
\newcommand{\mmsp}{mm-space}
\newcommand{\supp}{\mathop{\rm supp}}
\newcommand{\diam}{\mathop{\rm diam}}
\newcommand{\ObsDiam}{{\rm ObsDiam}}

\newcommand{\dx}{{d_X}}
\newcommand{\mux}{{\mu_X}}
\newcommand{\mmr}[1]{(\R,|\cdot|,{#1})}

\title{The maximum of the 1-measurement of a metric measure space}
\author{Hiroki Nakajima}
\date{\today}
\begin{document}
\maketitle
\begin{abstract}
For a metric measure space, we treat the set of distributions of 1-Lipschitz functions, which is called the 1-measurement.
On the 1-measurement, we have a partial order relation by the Lipschitz order introduced by Gromov \cite{Gmv:green}.
The aim of this paper is to study the maximum and maximal elements of the 1-measurement with respect to the Lipschitz order.
We present a necessary condition of a metric measure space for the existence of the maximum of the 1-measurement. We also consider a metric measure space that has the maximum of its 1-measurement. 
\end{abstract}
\tableofcontents
\section{Introduction}
In this paper, we study the maximum and the maximal elements of the 1-measurement of a metric measure space.
Let $(X,d_X)$ be a complete separable metric space with a Borel probability measure $\mux$. We call such a triple $(X,d_X,\mu_X)$ an {\it\mmsp}\ (metric measure space).
Based on the measure concentration phenomenon, M. Gromov introduced various concepts and invariants in the \mmsp\ framework \cite{Gmv:green}.
Observable diameter is one of the most important invariants defined by him. 
It is a quantity of how much the measure of an \mmsp\ concentrates and is defined by the 1-measurement. We assume that any \mmsp\ $X$ satisfies $X=\supp \mux$ unless otherwise stated, where $\supp\mux$ is the support of $\mux$.
The {\it 1-measurement} of an \mmsp\ $X$ is defined as
\[
\mathcal M(X;1):=\set{f_*\mux\mid f:X\to\R :\text{1-Lipschitz function}},
\]
where  a 1-Lipschitz function is a Lipschitz continuous function with its Lipschitz constant less than or equal to one. The 1-measurement has a natural order relation called the Lipschitz order (Definition \ref{def:Lip_ord} and Remark \ref{rem:Lip_ord}). 

We firstly treat the $n$-dimensional unit sphere $\Sn$ centered at the origin in $\R^{n+1}$ as an \mmsp. A compact Riemannian manifold is considered as an \mmsp\ with the Riemannian distance function and the normalized volume measure.
\begin{thm}[Gromov {\cite[\S 9]{Gmv:isop}}]\label{sn_max}
The push-forward $\xi_*\mu_\Sn$ of the measure $\mu_\Sn$ by the distance function $\xi$ from one point in $\Sn$ is the maximum of $\mathcal M(\Sn;1)$.
\end{thm}
We give a detailed proof of Theorem \ref{sn_max} in Section \ref{sec:sn_max}.
We use L{\'e}vy's isoperimetric inequality (Theorem \ref{thm:levy_ineq}) in the proof of Theorem \ref{sn_max}.
As a corollary of Theorem \ref{sn_max}, we see the normal law {\`a} la L{\'e}vy (Corollary \ref{normal_law}). This theorem can be thought as a finite-dimensional version of the normal law {\`a} la L{\'e}vy.

We obtain the following result for a general \mmsp. Denote the diameter of $X$ by $\diam X$.
\begin{prop}\label{diam_maximal}
Let $(X,\dx,\mux)$ be an \mmsp. Any measure $\mu\in\mathcal M(X;1)$ satisfying $\diam\supp\mu=\diam X<\infty$ is a maximal element of the 1-measurement $\mathcal M(X;1)$.
\end{prop}
Proposition \ref{diam_maximal} is simple and powerful to find a maximal element of the 1-measurement $\mathcal M(X;1)$.
As a corollary of Proposition \ref{diam_maximal}, we have the following.
\begin{cor}\label{xi_max}
Let an \mmsp\ $X$ satisfy $\diam X<\infty$ and a point $x_0\in X$ satisfy $\sup_{x\in X}\dx(x,x_0)=\diam X$. 
The push-forward $\xi_*\mux$ of $\mux$ by the distance function $\xi$ from the point $x_0$ is a maximal element of the 1-measurement $\mathcal M(X;1)$.
In particular, if the maximum of the 1-measurement $\mathcal M(X;1)$ exists, then it is $\xi_*\mux$.
\end{cor}
In the case where two points $x_0,x_1\in X$ satisfy $\sup_{x\in X}\dx(x,x_i)=\diam X,\ i=0,1$, each push-forward $(\xi_i)_*\mux$ of $\mux$ by the distance function $\xi_i$ from the point $x_i$ is a maximal element of the 1-measurement $\mathcal M(X;1)$.
Therefore, if $(\xi_0)_*\mux$ and $(\xi_1)_*\mux$ are not isomorphic to each other, then the 1-measurement $\mathcal M(X;1)$ has no maximum because it has two different maximal elements.
On the other hand, the push-forward by the distance function from one point does not depend on how to pick the point in a homogeneous space such as the flat torus $T^n\,(n\geq 2)$ or the projective space $\RPn\,(n\geq 2)$.
However, $\mathcal M(T^n;1)$ and $\mathcal M(\RPn;1)$ both have no maximum because of one of main theorems stated as follows.
\begin{thm}\label{max_threePt}
Assume that the 1-measurement $\mathcal M(X;1)$ has its maximum.
Then, for any two points $x,y\in X$ with $\dx(x,y)=\diam X$\\$<\infty$, we have
\[
\dx(x,z)+\dx(z,y)=\diam X \quad\text{for any point $z\in X$}.
\]
\end{thm}
We prove Theorem \ref{max_threePt} in Section \ref{sec:max_threePt}. 
Theorem \ref{max_threePt} is widely applicable not only for Riemannian manifolds but also for discrete spaces.

In the case where $X$ is a compact Riemannian homogeneous space, by using Theorem \ref{max_threePt}, we see that the cut locus of every point consists of a single point if $\mathcal M(X;1)$ has its maximum. Such a Riemannian manifold is called a Wiedersehen manifold and is known to be isometric to a round sphere $S^n(r)$ of radius $r>0$ \cite{Yang:odd}. Therefore, the following corollary follows.

\begin{cor}\label{homog_max}
Let $X$ be a compact Riemannian homogeneous space. Then, the 1-measurement $\mathcal M(X;1)$ has its maximum if and only if $X$ is isometric to a round sphere $S^n(r),\ r>0$.
\end{cor}


\section{Preliminaries}\label{preliminaries}
In this section, we enumerate some basics of \mmsp\ and prepare for describing the maximum and maximal elements of the 1-measurement.
We refer to \cite{Gmv:green,Shioya:mmg} for more details about this section.
\subsection{Some basics of \mmsp}\label{basics}
\begin{dfn}[\mmsp]
Let $(X,\dx)$ be a complete separable metric space with a Borel probability measure $\mux$. 
We call such a triple $(X,\dx,\mux)$ an {\it\mmsp}. 
We sometimes say that $X$ is an \mmsp, for which the metric and measure of $X$ are respectively indicated by $\dx$ and $\mux$.
\end{dfn}

We denote the Borel $\sigma$-algebra over $X$ by $\mathcal B_X$.   
For any point $x\in X$, any two subsets $A,B\subset X$ and any real number $r>0$, we define
\begin{align*}
d_X(x,A)&:=\inf_{y\in A} d_X(x,y),\\
d_X(A,B)&:=\inf_{x\in A,\,y\in B}d_X(x,y),\\
U_r(A)&:=\set{y\in X\mid d_X(y,A)<r},\\
B_r(A)&:=\set{y\in X\mid d_X(y,A)\leq r}.
\end{align*}

Let $p:X\to Y$ be a measurable map from a measure space $(X,\mu)$ to a topological space $Y$. {\it The push-forward of $\mu$ by the map $p$} is defined as $p_*\mu(A):=\mu(p^{-1}(A))$ for any $A\in \mathcal B_X$. 


\begin{dfn}[mm-isomorphism]
Two \mmsp s $X$ and $Y$ are said to be {\it mm-isomorphic} to each other if there exists an isometry $f:\supp\mu_X\to\supp\mu_Y$ such that $f_*\mu_X=\mu_Y$, where $\supp\mux$ is the {\it support of $\mux$}.
Such an isometry $f$ is called an {\it mm-isomorphism}.
The mm-isomorphism relation is an equivalence relation on the set of mm-spaces.
Denote by $\mathcal X$ the set of mm-isomorphism classes of \mmsp s.
\end{dfn}

Note that $X$ is mm-isomorphic to $(\supp\mux,\dx,\mux)$. We assume that any \mmsp\ $X$ satisfies \[X=\supp \mux\] unless otherwise stated.


\begin{dfn}[$1$-measurement]
The {\it $1$-measurement} $\mathcal M(X;1)$ of an \mmsp\ $X$ is defined as
\[
\mathcal M(X;1) := \{f_\ast\mu\mid f:X\to \mathbb R:\text{1-Lipschitz function}\}.
\]
\end{dfn}

\begin{dfn}[Lipschitz order]\label{def:Lip_ord}
Let $X$ and $Y$ be two \mmsp s.
We say that $X$ {\it dominates} $Y$ and write $Y\prec X$ if there exists a 1-Lipschitz map $f:X\to Y$ satisfying
\[
f_*\mu_X=\mu_Y.
\]
We call the relation $\prec$ on $\mathcal X$ the {\it Lipschitz order}.
\end{dfn}
\begin{prop}
The Lipschitz order $\prec$ is a partial order relation on $\mathcal X$.
\end{prop}
\begin{rem}\label{rem:Lip_ord}
Since an element $\mu$ of 1-measurement $\mathcal M(X;1)$ is a measure on the real line $\R$, the triple $\mmr\mu$ is an \mmsp.
We define the Lipschitz order between two elements of $\mathcal M(X;1)$ by considering $\mu\in\mathcal M(X;1)$ as an \mmsp\ in the above way.
In this manner, we consider the maximum and maximal elements of the 1-measurement $\mathcal M(X;1)$ with respect to the Lipschitz order.
For two measure $\mu,\nu\in\mathcal M(X;1)$, we write $\mu\prec\nu$ as 
$\mmr\mu\prec\mmr\nu$ for simplicity.
\end{rem}
\begin{rem}
For a Borel probability measure $\mu$ on the real line $\R$, we immediately see that the measure $\mu$ is the maximum of the 1-measurement $\mathcal M(\mmr\mu;1)$.
\end{rem}

\subsection{Observable diameter and partial diamter}
\begin{dfn}[Partial diameter]
Let $X$ be an \mmsp. 
For any real number $\alpha\in[0,1]$, we define the {\it partial diameter $\diam(X;\alpha)=\diam(\mu_X;\alpha)$ of $X$} as
\[
\diam(X;\alpha):=\inf\set{\diam A\mid \mu_X(A)\geq \alpha, \, A\in\mathcal B_X},
\]
where the {\it diameter of $A$} is defined by $\diam A:=\sup_{x,y\in A}d_X(x,y)$ for $A\neq\emptyset$ and $\diam \emptyset:=0$. 
\end{dfn}
\begin{dfn}[Observable diamter]
Let $X$ be an \mmsp.
For any real number $\kappa\in[0,1]$, we define the {\it $\kappa$-observable diameter \\
$\ObsDiam(X;-\kappa)$ of $X$} as
\[
\ObsDiam(X;-\kappa):=\sup_{\mu\in\mathcal M(X;1)}\diam(\mu;1-\kappa).
\]

\end{dfn}
\begin{prop}\label{prop:Lip_invariant}
Let $X$ and $Y$ be two \mmsp s and $\kappa\in[0,1]$ be a real number.
If we have $Y\prec X$, then we obtain
\begin{align*}
\diam(Y;1-\kappa)&\leq\diam(X;1-\kappa),\\
\ObsDiam(Y;-\kappa)&\leq\ObsDiam(X;-\kappa).
\end{align*}
\end{prop}
In other words, the partial diameter and the $\kappa$-observable diameter are non-decreasing invariants with respect to the Lipschitz order.

\subsection{L{\'e}vy's isoperimetric inequality}
Let $\Sn$ be the $n$-dimensional unit sphere centered at the origin in the (n+1)-dimensional Euclidean space $\R^{n+1}$.
We assume the distance $d_\Sn(x,y)$ between two points $x$ and $y$ in $\Sn$ to be the geodesic distance and the measure $\mu_\Sn$ on $\Sn$ to be the Riemannian volume measure on $\Sn$ normalized as $\mu_\Sn(\Sn)=1$.
Then, $(\Sn,\,d_\Sn,\,\mu_\Sn)$ is an \mmsp.

\begin{thm}[L{\'e}vy's isoperimetric inequality \cite{Levy:iso,Milman:iso}]\label{thm:levy_ineq}
For any closed subset $\Omega\subset \Sn$, we take a metric ball $B_\Omega$ of $\Sn$ with $\mu_\Sn(B_\Omega)=\mu_\Sn(\Omega)$.
Then we have
\[
\mu_\Sn(U_r(\Omega))\geq\mu_\Sn(U_r(B_\Omega))
\]
for any $r>0$.
\end{thm}
\subsection{Box distance}
In this subsection, we briefly describe the box distance which is needed in subsection \ref{sec:normal_law}.
\begin{dfn}[Parameter]
Let $I:=[0,1)$ and $\mathcal L^1$ be the one-dimensional Lebesgue measure on $I$.
Let $X$ be a topological space with a Borel probability measure $\mux$.
A map $\varphi:I\to X$ is called a {\it parameter of $X$} if $\varphi$ is a Borel-measurable map such that
\[
\varphi_*\mathcal L^1=\mux.
\]
\end{dfn}
\begin{dfn}[Pseudo-metric]
A {\it pseudo-metric $\rho$ on a set $S$} is defined to be a function $\rho:S\times S\to[0,\infty)$ satisfying that, for any $x,y,z\in S$,
\begin{enumerate}
\item $\rho(x,x)=0,$
\item $\rho(y,x)=\rho(x,y),$
\item $\rho(x,z)\leq\rho(x,y)+\rho(y,z)$.
\end{enumerate}
\end{dfn}
\begin{dfn}[Box distance]
For two pseudo-metrics $\rho_1$ and $\rho_2$ on $I$, we define $\Box(\rho_1,\rho_2)$ to be the infimum of $\varepsilon\geq 0$ satisfying that there exists a Borel subset $I_0\subset I$ such that
\begin{enumerate}
\item $|\rho_1(s,t)-\rho_2(s,t)|\leq\varepsilon$ for any $s,t\in I_0$,
\item $\mathcal L^1(I_0)\geq 1-\varepsilon$.
\end{enumerate}
We define the {\it box distance $\Box(X,Y)$ between two \mmsp s $X$ and $Y$} to be the infimum of $\Box(\varphi^*d_X,\psi^*d_Y)$, where $\varphi:I\to X$ and $\psi:I\to Y$ run over all parameters of $X$ and $Y$, respectively, and where $\varphi^*\dx(s,t):=\dx(\varphi(s),\varphi(t))$ for $s,t\in I$.
\end{dfn}
\begin{thm}
The box distance $\Box$ is a metric on the set $\mathcal X$ of mm-isomorphism classes of \mmsp s.
\end{thm}
\begin{prop}\label{prop:box_prok}
Let $X$ be a complete separable metric space. For any two Borel probability measures $\mu$ and $\nu$ on $X$, we have,
\[
\Box((X,\mu),(X,\nu))\leq 2d_{\mathrm P}(\mu,\nu),
\]
where $d_{\mathrm P}$ is the Prohorov distance.
\end{prop}
\begin{thm}\label{thm:lip_box}
Let $X,Y,X_n$ and $Y_n$ be \mmsp s, $n=1,2,\dots$. 
If $X_n$ and $Y_n$ $\Box$-converge to $X$ and $Y$ respectively as $n\to\infty$ and if $X_n\prec Y_n$ for any $n$, then $X\prec Y$.
\end{thm}

\section{The maximum of the 1-measurement of $n$-dimensional sphere}
\subsection{The maximum of the 1-measurement of $n$-dimensional sphere --The proof of Theorem \ref{sn_max}--}\label{sec:sn_max}
The aim of this subsection is to prove Theorem \ref{sn_max}.
We prepare some lemmas for the proof.
\begin{lem}\label{lem_sn_max}
Let $X$ be an \mmsp\ and $f: X\to\mathbb R$ be a Borel measurable function.
We define the function $F:\mathbb R\to [0,1]$ as $F(t):=f_*\mu_X((-\infty,t])$.
If the function $F|_{\Image f}:\Image f\to[0,1]$ is bijective, we have
\[
F_*f_*\mu_X((-\infty,a])=a
\]
for all $a\in [0,1]$.
\end{lem}
\begin{proof}
For any $a\in[0,1]$, we see
\begin{align*}
F_*f_*\mu_X((-\infty,a])&=f_*\mu_X(F^{-1}((-\infty,a]))\\
&=f_*\mu_X(\set{t\in\mathbb R\mid F(t)\leq a})\\
&=f_*\mu_X(\set{t\in\Image f\mid F|_{\Image f}(t)\leq a})\\
&=f_*\mu_X(\set{t\in\Image f\mid t\leq (F|_{\Image f})^{-1}(a)})\\
&=f_*\mu_X(\set{t\in\mathbb R\mid t\leq (F|_{\Image f})^{-1}(a)})\\
&=f_*\mu_X((-\infty,(F|_{\Image f})^{-1}(a)])\\
&=F((F|_{\Image f})^{-1}(a))=a,
\end{align*}
where we use the non-decreasing and bijective property of $F|_{\Image f}$ in the fourth equality.
This completes the proof.
\end{proof}
\begin{lem}\label{tilde_bdd}
For a non-decreasing function $G:\mathbb R\to[0,1]$ with $G(t_0)=0$ for some $t_0\in\mathbb R$, 
we define $\tilde G:(0,1]\to\mathbb R$ by
\[
\tilde G(s):=\inf\set{t\in\mathbb R\mid s\leq G(t)}.
\]
Then, $\tilde G$ is non-decreasing and lower bounded on $(0,1]$.
In particular, $\tilde G$ takes finite values on $(0,1]$.
\end{lem}
\begin{proof}
We take a real number $t_0\in\R$ satisying $G(t_0)=0$.
Fix a real number $s\in(0,1]$ and define $A:=\set{t\in\mathbb R\mid s\leq G(t)}$.
For any element $t\in A$, we have $G(t_0)<s\leq G(t)$.
Since $G$ is non-decreasing, the inequality $t_0<t$ follows.
This implies that $t_0\leq\tilde G(s)$.
The function $\tilde G$ is a non-decreasing function on $(0,1]$ because we have $\set{t\in\mathbb R\mid s'\leq G(t)}\supset\set{t\in\mathbb R\mid s\leq G(t)}$ for any $0<s'\leq s$.
This completes the proof.
\end{proof}
\begin{lem}\label{tilde_lem}
Let $G:\mathbb R\to[0,1]$ be a non-decreasing and right continuous function such that $G(t_0)=0$ for some $t_0\in\mathbb R$.
We define $\tilde G:[0,1]\to\mathbb R$ by
\[
\tilde G(s):=
\begin{cases}
\inf\set{t\in\mathbb R\mid s\leq G(t)} & \text{if $s\in(0,1]$},\\
c & \text{if $s=0$},
\end{cases}
\]
where $c$ is an arbitrary constant. Then, we have 
\begin{eqnarray}
G\circ\tilde G(s)\geq s, & s\in [0,1],\label{tilde_geq}\\
\tilde G\circ G(t)\leq t, & \text{$t\in\mathbb R$ with $G(t)>0,\label{tilde_leq}$}\\
\tilde G^{-1}((-\infty,t])\setminus\{0\}=(0,G(t)],& t\in\mathbb R.\label{tilde_set}
\end{eqnarray}
\end{lem}
\begin{proof}
We prove \eqref{tilde_geq}. If $s=0$, we have \eqref{tilde_geq} because $\Image G\subset [0,1]$.
Fix a real number $s\in(0,1]$ and define $A:=\set{t\in\mathbb R\mid s\leq G(t)}$.
By the definition of infimum, we have
\[
G(t')\geq \inf_{t\in A}G(t)
\]
for any $t'\in A$.
For any $t'>\inf A$, we have $t'\in A$ because $G$ is non-decreasing.
By this, we have
\[
\lim_{t'\to \inf A+0}G(t')\geq \inf_{t\in A}G(t).
\]
We obtain
\[G(\inf A)\geq \inf_{t\in A}G(t)\]
by the right continuity of $G$.
Therefore, we have
\begin{align*}
G(\tilde G(s))&=G(\inf A)\\
&\geq \inf_{t\in A}G(t)\\
&=\inf\set{G(t)\mid s\leq G(t)}\\
&\geq s.
\end{align*}

We prove \eqref{tilde_leq}.
We take any real number $t\in \mathbb R$ satisfying $G(t)>0$, then we have
\[
\tilde G(G(t))=\inf\set{t'\in\mathbb R\mid G(t')\geq G(t)}\leq t.
\]

We prove \eqref{tilde_set}.
Take any real number $s\in \tilde G^{-1}((-\infty,t])\setminus\{0\}$.
Since $\tilde G^{-1}(\mathbb R)=[0,1]$, we have $s\in (0,1]$.
It follows from $\tilde G(s)\leq t$ and the non-decreasing property of $G$ that $G\circ \tilde G(s)\leq G(t)$.
This implies that $s\leq G(t)$ by \eqref{tilde_geq} and we have $s\in(0,G(t)]$.
Conversely, take any real number $s\in(0, G(t)]$.
We obtain $\tilde G(s)\leq \tilde G\circ G(t)$ because $G$ is non-decreasing by Lemma \ref{tilde_bdd}.
Then, we have $\tilde G(s)\leq t$ by \eqref{tilde_leq}.
This completes the proof.
\end{proof}
\begin{rem}\label{tilde_borel}
In Lemma \ref{tilde_lem}, $\tilde G$ is a Borel measurable function.
In fact, $\tilde G|_{(0,1]}$ is lower semi-continuous because we have $(\tilde G|_{(0,1]})^{-1}((-\infty,t])=\tilde G^{-1}((-\infty,t])\setminus\{0\}=(0,G(t)]$ and $(0,G(t)]$ is a closed subset in  $(0,1]$.
\end{rem}
\begin{lem}\label{meas_pres}
Let $f,g:X\to\mathbb R$ be two Borel measurable functions and define two functions $F,G:\mathbb R\to[0,1]$ as $F(t):=f_*\mu_X((-\infty,t]),\ G(t):=g_*\mu_X((-\infty,t])$.
We assume that some $t_0$ satisfies $G(t_0)=0$.
We define $\tilde G:[0,1]\to\mathbb R$ by
\[
\tilde G(s):=
\begin{cases}
\inf\set{t\in\mathbb R\mid s\leq G(t)} & \text{if $s\in(0,1]$},\\
c & \text{if $s=0$},
\end{cases}
\]
where $c$ is an arbitrary constant. 
We define $\varphi:\mathbb R\to\mathbb R$ by $\varphi:=\tilde G\circ F$.
If $F|_{\Image f}:\Image f\to [0,1]$ is bijective, we have 
\[
\varphi_*f_*\mu_X=g_*\mu_X.
\]
\end{lem}
\begin{proof}
Take any real number $t\in\mathbb R$. we have
\begin{align*}
\varphi_*f_*\mu_X((-\infty,t])&=\tilde G_*F_*f_*\mu_X((-\infty,t])\\
&=F_*f_*\mu_X(\tilde G^{-1}((-\infty,t]))\\
&=F_*f_*\mu_X(\tilde G^{-1}((-\infty,t])\setminus\{0\})\\
&=F_*f_*\mu_X((0,G(t)])\\
&=F_*f_*\mu_X((-\infty,G(t)])\\
&=G(t)\\
&=g_*\mu_X((-\infty,t]).
\end{align*}
In the third and fourth equality, we use $F_*f_*\mu_X((-\infty,0])=0$ obtained by Lemma \ref{lem_sn_max}. 
We use \eqref{tilde_set} of Lemma \ref{tilde_lem} in the fourth equality.
We have the sixth equality by Lemma \ref{lem_sn_max}.
This completes the proof.
\end{proof}

\begin{proof}[Proof of Theorem \ref{sn_max}]
Fix a point $\bar x\in\Sn$ and define $\xi:\Sn\to \mathbb R$ by $\xi(x):=d_\Sn(\bar x,x)$.
Take any 1-Lipschitz function $g:\Sn\to\mathbb R$.
We prove the existence of a 1-Lipschitz function $\varphi:\mathbb R\to\mathbb R$ satisfying
\[
\varphi_*\xi_*\mu_X=g_*\mu_X
\]
in the following.
Put two functions $V,G:\mathbb R\to[0,1]$ as $V(t):=\xi_*\mu_\Sn((-\infty,t]),\ G(t):=g_*\mu_\Sn((-\infty,t])$.
We define $\tilde G:[0,1]\to\mathbb R$ as
\[
\tilde G(s):=\inf\set{t\in\mathbb R\mid s\leq G(t)}
\]
if $s\in(0,1]$, and 
\[
\tilde G(0)=\lim_{s\to +0}\tilde G(s),
\]
if $s=0$.
We have $G(t_0)=0$ for some $t_0$ because $g$ have a lower bound.
The existence of limit is guaranteed because $G$ is non-decreasing and $\tilde G$ has a lower bound on $(0,1]$ by Lemma \ref{tilde_bdd}.
Put $\varphi:\mathbb R\to\mathbb R$ as $\varphi:=\tilde G\circ V$.
We apply Lemma \ref{meas_pres} to obtain
\[
\varphi_*\xi_*\mu_X=g_*\mu_X
\]
since $V|_{\Image \xi}$ is bijective.

Let us prove that $\varphi$ is a 1-Lipschitz function.
If $t\leq 0$, we have $\varphi(t)=\tilde G(0)$ by $V(t)=0$.
We obtain
\[
\lim_{t\to +0}\varphi(t)=\lim_{t\to +0}\tilde G\circ V(t)=\tilde G(0)
\]
because $\tilde G$ is continuous at $0$ and $\lim_{t\to +0}V(t)=0$.
By this, we prove $\varphi$ is a 1-Lipschitz function in the case where $t>0$.
The function $\varphi$ is non-decreasing since two functions $\tilde G, V$ are both non-decreasing.
Thus, it is sufficient to prove that $\varphi(t+\varepsilon)\leq \varphi(t)+\varepsilon$ for any $\varepsilon>0$.
Fix $t>0$ and take any $\varepsilon>0$. We have
\begin{align*}
\mu_\Sn(B_t(\bar x))&=\xi_*\mu_\Sn((-\infty,t])\\
&=V(t)\\
&\leq (G\circ \tilde G)(V(t))\\
&=G\circ \varphi (t)\\
&=\mu_\Sn(g^{-1}((-\infty,\varphi(t)])),
\end{align*}
where we use \eqref{tilde_geq} of Lemma \ref{tilde_lem} in the inequality on the third line.
We obtain
\[
\mu_\Sn(B_{t+\varepsilon}(\bar x))\leq\mu_\Sn(B_\varepsilon(g^{-1}((-\infty,\varphi(t)])))
\]
by applying Theorem \ref{thm:levy_ineq} (L{\'e}vy's isoperimetric inequality).
We use this inequality to obtain
\begin{align*}
V(t+\varepsilon)&=\xi_*\mu_\Sn((-\infty,t+\varepsilon])\\
&=\mu_\Sn(B_{t+\varepsilon}(\bar x))\\
&\leq \mu_\Sn(B_\varepsilon(g^{-1}((-\infty,\varphi(t)])))\\
&\leq\mu_\Sn(g^{-1}(B_\varepsilon((-\infty,\varphi(t)])))\\
&=g_*\mu_\Sn((-\infty,\varphi(t)+\varepsilon])\\
&=G(\varphi(t)+\varepsilon),
\end{align*}
where we have the inequality on the fourth line because $g$ is a 1-Lipschitz function.
Therefore, we have
\begin{align*}
\varphi(t+\varepsilon)&=\tilde G\circ V(t+\varepsilon)\\
&\leq \tilde G\circ G(\varphi(t)+\varepsilon)\\
&\leq\varphi(t)+\varepsilon,
\end{align*}
where we use \eqref{tilde_leq} of Lemma \ref{tilde_lem} in the inequality of the third line.
This completes the proof.
\end{proof}

\subsection{The relation between the normal law {\`a} la L{\'e}vy and Theorem \ref{sn_max}}\label{sec:normal_law}
The aim of this section is to prove Corollary \ref{normal_law} by Theorem \ref{sn_max}.
\begin{cor}[Normal law {\`a} la L{\'e}vy \cite{Gmv:green,Shioya:mmg}]\label{normal_law}
Let $f_n:\Srn\to\mathbb R,\quad n=1,2,\dots ,$ be 1-Lipschitz functions. Assume that a subsequence $\{f_{n_i}\}$ of $\{f_n\}$ satisfy that the push-forward $(f_{n_i})_*\mu_\Srn$ converges weakly to a Borel probability measure $\sigma$. Then we have
\[
(\mathbb R,|\cdot|,\sigma)\prec(\mathbb R,|\cdot|,\gamma^1).
\]
\end{cor}

We prepare some lemmas to prove Corollary \ref{normal_law}.

\begin{lem}\label{cos_lim}
For any real number $r\in\mathbb R$, we have
\[
\cos^{n-1}\frac r{\sqrt n} \to e^{-\frac{r^2}2} \text{\quad as $n\to\infty$}.
\]
\end{lem}
\begin{proof}
If $r=0$, then the lemma is trivial. Assume $r\neq 0$.
We first prove $\liminf_{n\to\infty}\cos^{n-1}\frac r{\sqrt n}\geq e^{-\frac{r^2}2}$.
We use $\cos x\geq 1-\frac{x^2}2$ for any $x\in[-\frac\pi 2,\frac\pi 2]$.
Fix a real number $r\in \mathbb R\setminus\{0\}$. For some positive integer $N\in\mathbb N$, we have $r\in[-\frac\pi 2\sqrt n,\frac\pi 2\sqrt n]$ for any positive integer $n\geq N$.
Then, we have
\[
\cos^{n-1}\frac r{\sqrt n}\geq \left(1-\frac 12\left(\frac r{\sqrt n}\right)^2\right)^{n-1}
\]
for any positive integer $n\geq N$.
We obtain $\liminf_{n\to\infty}\cos^{n-1}\frac r{\sqrt n}\geq e^{-\frac{r^2}2}$ because we have
\[
\left(1-\frac 12\left(\frac r{\sqrt n}\right)^2\right)^{n-1}
=\left(1-\frac{r^2}{2n}\right)^{\left(-\frac{2n}{r^2}\right)\cdot\left(-\frac{r^2}2\right)-1}
\to e^{-\frac{r^2}2}\text{\quad as $n\to\infty$}.
\]

We next prove $\limsup_{n\to\infty}\cos^{n-1}\frac r{\sqrt n}\leq e^{-\frac{r^2}2}$.
Fix a real number $r\in\mathbb R\setminus\{0\}$ and take any real number $\varepsilon\in(0,1)$.
Since $\lim_{x\to 0}\frac{1-\cos x}{x^2}=\frac 12$, for some $\delta>0$, we have $\cos x\leq 1-\left(\frac 12-\varepsilon\right)x^2$ for any $x\in(-\delta,\delta)$.
We take some positive integer $N\in\mathbb N$ satisfying $|\frac r{\sqrt N}|<\delta$.
For any positive integer $n\geq N$, we have 
\[
\cos^{n-1}\frac r{\sqrt n}\leq \left(1-\frac{1-\varepsilon}2\left(\frac r{\sqrt n}\right)^2\right)^{n-1}.
\]
Since we have
\begin{align*}
\left(1-\frac{1-\varepsilon}2\left(\frac r{\sqrt n}\right)^2\right)^{n-1}&=
\left(1-\frac{1-\varepsilon}2\cdot\frac{r^2}n\right)^{\left(-\frac{2n}{(1-\varepsilon)r^2}\right)\cdot\left(-\frac{(1-\varepsilon)r^2}2\right)-1}\\
&\to e^{-\frac{r^2}2\cdot(1-\varepsilon)}\text{\quad as $n\to\infty$}
\end{align*}
and
$e^{-\frac{r^2}2\cdot(1-\varepsilon)}\to e^{-\frac{r^2}2}$\text{\quad as }$\varepsilon\to +0$, we obtain $\limsup_{n\to\infty}\cos^{n-1}\frac r{\sqrt n}\leq e^{-\frac{r^2}2}$.
This completes the proof.
\end{proof}
\begin{lem}\label{cos_bdd}
For any integer $n\geq 2$ and any real number $r\in[-\frac\pi 2\sqrt n,$\\
$\frac\pi 2\sqrt n]$, we have
\[
\cos^{n-1}\frac r{\sqrt n}\leq e^{-\frac{r^2}4}
\]
\end{lem}
\begin{proof}
Take any integer $n\geq 2$.
This lemma is clear if $r=\pm\frac\pi 2\sqrt n$.
Then, we prove the lemma in the case $r\in(-\frac\pi 2\sqrt n,\frac\pi 2\sqrt n)$.
By the symmetry, we may assume $r\geq 0$.
Setting
\[
f(r):=-\frac{r^2}4-(n-1)\log\cos\frac r{\sqrt n},
\]
we have
\begin{align*}
f'(r)&=-\frac r2+(n-1)\cdot\frac 1{\sqrt n}\tan\frac r{\sqrt n}\\
&=-\frac{\sqrt n}2\cdot\frac r{\sqrt n}+\left(\sqrt n-\frac 1{\sqrt n}\right)\tan\frac r{\sqrt n}\\
&\geq\frac{\sqrt n}2\left(\tan\frac r{\sqrt n}-\frac r{\sqrt n}\right)\geq 0,
\end{align*}
where we use $n\geq 2$ in the first inequality and $\frac r{\sqrt n}\in[0,\frac\pi 2)$ in the second inequality.
Since $f(0)=0$, we obtain $f(r)\geq 0$ for any $r\in[0,\frac\pi 2\sqrt n)$.
This completes the proof.
\end{proof}
\begin{lem}\label{one_point_weak}
Fix a point $\bar x\in\Srn$, and put $\xi_n:\Srn\to\mathbb R$ as $\xi_n(x):=d_\Srn(x,\bar x)$.
Then, we have
\[
\frac{d((\xi_n-\sqrt n\frac\pi 2)_*\mu_\Srn)}{d\mathcal L^1}(r)\to
 \frac{d\gamma^1}{d\mathcal L^1}(r),\quad n\to \infty
\]
for any $r\in \mathbb R$, where we define $\gamma^1$ as
\[
\frac{d\gamma^1}{d\mathcal L^1}(r):=\frac 1{\sqrt{2\pi}}e^{-\frac{r^2}2}.
\]
In particular, we have
\[
(\xi_n-\sqrt n\frac\pi 2)_*\mu_\Srn\to \gamma^1\text{\quad as $n\to\infty$ weakly}.
\]
\end{lem}
\begin{proof}
Let $\chi_{[-\frac\pi 2\sqrt n,\frac\pi 2\sqrt n]}$ be the indicator function of the subset $[-\frac\pi 2\sqrt n,$\\
$\frac\pi 2\sqrt n]$.
We have
\[
\frac{d((\xi_n-\sqrt n\frac\pi 2)_*\mu_\Srn)}{d\mathcal L^1}(r)=\chi_{[-\frac\pi 2\sqrt n,\frac\pi 2\sqrt n]}\cdot\frac{\cos\frac r{\sqrt n}}{\int_{-\frac\pi 2\sqrt n}^{\frac\pi 2\sqrt n}\cos^{n-1}\frac t{\sqrt n}dt}
\]
for any real number $r\in\mathbb R$.
We obtain 
\[
\chi_{[-\frac\pi 2\sqrt n,\frac\pi 2\sqrt n]}\cdot\frac{\cos\frac r{\sqrt n}}{\int_{-\frac\pi 2\sqrt n}^{\frac\pi 2\sqrt n}\cos^{n-1}\frac t{\sqrt n}dt}
\to \frac{e^{-\frac{r^2}2}}{\int_\mathbb R e^{-\frac{t^2}2}dt}\text{\quad as $n\to\infty$}
\]
because of Lebesgue's dominated convergence theorem, Lemma \ref{cos_bdd} and Lemma \ref{cos_lim}.
This completes the proof.
\end{proof}
\begin{proof}[Proof of Corollary \ref{normal_law}]
Take any 1-Lipschitz functions $f_n:\Srn\to \mathbb R,\quad n=1,2,\dots$. We may assume $\Box(\mmr{(f_{n_i})_*\mu_{S^n(\sqrt n_i)}},\mmr\sigma)\to 0\text{\quad as $n\to\infty$}$ because of Proposition \ref{prop:box_prok}.
Fix a point $\bar x\in\Srn$ and define $\xi_n(x):=d_\Srn(x,\bar x)$. By applying Theorem \ref{sn_max}, we have $(\mathbb R,|\cdot|,(f_n)_*\mu_\Srn)\prec(\mathbb R,|\cdot|,(\xi_n)_*\mu_\Srn)$ for any positive integer $n\in\mathbb N$.
Since we have $\Box(\mmr{(\xi_n)_*\mu_\Srn},\mmr{\gamma^1})\to 0$ by Lemma \ref{one_point_weak}, we obtain $(\mathbb R,|\cdot|,\sigma)\prec(\mathbb R,|\cdot|,\gamma^1)$ by Theorem \ref{thm:lip_box}.
This completes the proof.
\end{proof}

\section{A necessary condition for the existence of the maximum of the 1-measurement}
\subsection{Maximal elements of the 1-measurement}
\begin{lem}\label{diam_iso}
Let $\mu,\nu$ be two Borel probability measures on $\R$.
We assume $\diam\supp\mu=\diam\supp\nu<\infty$ and $\mu\prec\nu$.
Then, two \mmsp s $\mmr\mu$ and $\mmr\nu$ are mm-isomorphic to each other.
\end{lem}
\begin{proof}
Since $\mu\prec\nu$, there exist a 1-Lipschitz function $\varphi:\supp\nu\to\supp\mu$ such that $\varphi_*\nu=\mu$.
Put $c:=\diam\supp\mu=\diam\supp\nu$, $y_0:=\min\supp\mu$ and $y_1:=\max\supp\mu$. We have $c=y_1-y_0$.
Since $\varphi$ is surjective, there exist $x_i\in\supp\nu$ such that $\varphi(x_i)=y_i$ for each $i=0,1$.
We have
\[
c=y_1-y_0=\varphi(x_1)-\varphi(x_0)\leq|x_1-x_0|\leq c
\]
because $\varphi$ is a 1-Lipschitz function.
Therefore we obtain $|x_1-x_0|= c$.
In particular, the point $x_0$ is the maximum or the minimum of $\supp \nu$.
We prove $\varphi(x)=y_0+|x-x_0|$ for any $x\in\supp\nu$ in the following.
If we prove it, we see that $\varphi$ is isometry and the proof is completed.
Let us prove it in the case where $x_0\leq x_1$.
We have $x_0=\min\supp\nu,\ x_1=\max\supp\nu$ because $x_1=x_0+c$.
We obtain
\[
\varphi(x)-y_0=\varphi(x)-\varphi(x_0)\leq x-x_0
\]
for any $x\in\supp\nu$.
This implies that
$\varphi(x)\leq y_0+|x-x_0|$.
We have
\begin{align*}
(y_0+c)-\varphi(x)&=y_1-\varphi(x)\\
&=\varphi(x_1)-\varphi(x)\\
&\leq x_1-x\\
&=x_0+c-x.
\end{align*}
We also have
$\varphi(x)\geq y_0+|x-x_0|$.
We prove it in the case where $x_0\geq x_1$ similarly.
In fact, we have
$\varphi(x)\leq |x-x_0|+y_0$
because
\[
\varphi(x)-y_0=\varphi(x)-\varphi(x_0)\leq |x-x_0|
\]
and we have $\varphi(x)\geq |x-x_0|+y_0$ because
\begin{align*}
x_0-x_1+y_0-\varphi(x)&=c+y_0-\varphi(x)\\
&=y_1-\varphi(x)\\
&= \varphi(x_1)-\varphi(x)\\
&\leq x-x_1.
\end{align*}
This completes the proof.
\end{proof}
\begin{proof}[Proof of Proposition \ref{diam_maximal}]
Take a measure $\mu\in\mathcal M(X;1)$ with $\diam\supp$\\
$\mu=\diam X$ and a measure $\nu\in\mathcal M(X;1)$ with $\mu\prec\nu$.
We have $\diam\supp\nu\leq\diam X$ because $\nu\in\mathcal M(X;1)$.
We also have $\diam\supp\mu$\\
$\leq\diam\supp\nu$ because $\mu\prec\nu$.
Two \mmsp s $\mmr\mu$ and $\mmr\nu$ is mm-isomorphic to each other because of $\diam\supp\mu=\diam\supp\nu$ and Lemma \ref{diam_iso}.
This completes the proof.
\end{proof}
\subsection{A necessary condition for the existence of the maximum of the 1-measurement}\label{sec:max_threePt}
\begin{proof}[Proof of Theorem \ref{max_threePt}]
We prove the contrapositive proposition.
Take three points $x_0,x_1,x_2\in X$ satisfying $\dx(x_0,x_1)=\diam X$ and $\dx(x_0,$\\
$x_2)+\dx(x_2,x_1)>\dx(x_0,x_1)$.
Put $r_i:=\dx(x_i,x_2),\ i=0,1$, $R:=\diam X$ and $D:=\frac{r_0+r_1-R}2>0$.
We have $r_i-D>0,\ i=0,1$ and $(r_0-D)+(r_1-D)=R$.
By the symmetry, we may assume $r_1\leq r_0$.
Put a function $\xi:X\to\R$ as $\xi(x):=\dx(x,x_0)$ and define a 1-Lipschitz function $\zeta:X\to\R$ by
\[
\zeta(x):=
\begin{cases}
\dx(x,x_0) & \text{if $x\in U_{r_0-D}(x_0)$},\\
R-\dx(x,x_1) & \text{if $x\in U_{r_1-D}(x_1)$},\\
r_0-D & \text{otherwise}.
\end{cases}
\]
Let us prove that $\zeta$ is a 1-Lipschitz function.
In the case where $x\in U_{R-r_1}(x_0)$ and $y\in U_{r_1}(x_1)$, we have
\begin{align*}
|\zeta(x)-\zeta(y)|&=|\dx(x,x_0)-R+\dx(y,x_1)|\\
&=R-\dx(x,x_0)-\dx(y,x_1)\\
&=\dx(x_0,x_1)-\dx(x,x_0)-\dx(y,x_1)\\
&\leq \dx(x,y).
\end{align*}
In the case where $x\in U_{R-r_1}(x_0)$ and $y\in U_{R-r_1}(x_0)^c\cap U_{r_1}(x_1)^c$, we have
\begin{align*}
|\zeta(x)-\zeta(y)|&=|\dx(x,x_0)-(R-r_1)|\\
&=-\dx(x,x_0)+R-r_1\\
&\leq -\dx(x,x_0)+\dx(x_0,y)\\
&\leq \dx(x,y).
\end{align*}
In the case where $x\in U_{r_1}(x_1)$ and $y\in U_{r_1}(x_1)^c\cap U_{r_1}(x_1)^c$, we have
\begin{align*}
|\zeta(x)-\zeta(y)|&=|(R-r_1)-(R-\dx(x,x_1))|\\
&=|\dx(x,x_1)-r_1|\\
&=r_1-\dx(x,x_1)\\
&\leq\dx(x_1,y)-\dx(x,x_1)|\\
&\leq\dx(x,y).
\end{align*}
Thus, the function $\zeta$ is a 1-Lipschitz function.

Two measures $\xi_*\mux$ and $\zeta_*\mux$ are both maximal elements by Proposition \ref{diam_maximal} and $\diam\supp\xi_*\mux=\diam\supp\zeta_*\mux=\diam X$.
Let us prove that two measures $\xi_*\mux$ and $\zeta_*\mux$ are not mm-isomorphic to each other.
It is sufficient to prove that $\xi_*\mux\neq\zeta_*\mux$ and $(R-\xi)_*\mux\neq\zeta_*\mux$.
We prove those by contradiction. We first assume $\xi_*\mux=\zeta_*\mux$.
We have
\begin{align*}
\mux(B_{r_0-D}(x_0)\sqcup U_D(x_2))&\leq\mux(U_{r_1-D}(x_1)^c)\\
&=\zeta_*\mux([0,r_0-D])\\
&=\xi_*\mux([0,r_0-D])\\
&=\mux(B_{r_0-D}(x_0)).
\end{align*}
This inequality contradicts $\mux(U_D(x_2))>0$.
We next assume $(R-\xi)_*\mux=\zeta_*\mux$ and we have
\begin{align*}
\mux(U_{r_1-D}(x_1))\sqcup U_D(x_2))&\leq\mux(B_{r_0-D}(x_0)^c)\\
&=\xi_*\mu((r_0-D,R])\\
&=(R-\xi)_*\mux([0,r_1-D))\\
&=\zeta_*\mux([0,r_1-D))\\
&=\xi_*\mux([0,r_1-D))\\
&=(R-\xi)_*\mux((r_0-D,R])\\
&=\zeta_*\mux((r_0-D,R])\\
&=\mux(U_{r_1-D}(x_1)),
\end{align*}
where we use $r_1-D\leq r_0-D$ in the equality on the fifth line.
This inequality contradicts $\mux(U_D(x_2))>0$. 
This completes the proof.
\end{proof}
\section*{Acknowledgments}
The author would like to thank Prof.\,Takashi Shioya for many helpful suggestions. He also thanks Daisuke Kazukawa and Yuya Higashi for many stimulating conversations. He is grateful to Prof.\,Kei Funano for providing important comments.

\end{document}